\documentclass[11pt,reqno]{article}
\usepackage{bbding}
\usepackage{pifont}
\usepackage{amsmath}
\usepackage{mathrsfs}
\usepackage{amsfonts,bm}
\usepackage{amsthm}
\usepackage{epsfig}
\usepackage[]{harpoon}
\usepackage[active]{srcltx}
\usepackage{indentfirst,latexsym}
\usepackage{psfrag}
\usepackage[]{amssymb}
\usepackage{cite}
\usepackage{color}
\usepackage{caption}

\newcommand{\new}{\newcommand*}\new{\rnew}{\renewcommand*}
\new{\newe}{\newenvironment*}\new{\stl}{\setlength}
\stl{\textwidth}{155mm}\stl{\textheight}{22cm}\stl{\headheight}{0cm}
\stl{\topmargin}{0cm}\stl{\oddsidemargin}{0.5cm}\stl{\evensidemargin}{0cm}
\rnew{\arraystretch}{1.2}\rnew{\baselinestretch}{1.2}
\renewcommand{\thefootnote}{\ding{73}}
\newtheorem{thm}{Theorem}[section]

\newtheorem{lem}{Lemma}[section]

\newcommand{\dps}{\displaystyle}

\newcommand{\fr}{\frac}

\newcommand{\pa}{\partial}

\numberwithin{equation}{section}

\newe{keywords}
   {\begin{quote}{\bf Keywords:}}
      {\end{quote}}
\newe{AMS}
   {\begin{quote}{\bf AMS subject classification 2010:}}
      {\end{quote}}
\newe{MSC}{\vspace*{5mm}
    {\noindent\it Mathematics Subject Classification(2010):}}{}
\new{\sect}[1]{\section{#1}\setcounter{equation}{0}
 \setcounter{thm}{0}\setcounter{lmm}{0}\setcounter{rmk}{0} }

\begin{document}
\title{On the degenerate Cauchy problem for a nonlinear variational wave system, Part I: The same wave speed case
}

\author{
Yanbo Hu$^*$, Huijuan Song
\\{\small \it Department of Mathematics, Hangzhou Normal University,
Hangzhou, 311121, PR China}}

\rnew{\thefootnote}{\fnsymbol{footnote}}

\footnotetext{ $^*$Corresponding author. }

\footnotetext{ Email address: yanbo.hu@hotmail.com (Y. Hu), huijuan\_mail@163.com (H. Song). }

\date{}

\maketitle

\begin{abstract}

We investigate a one-dimensional nonlinear wave system which arises from a variational principle modeling a type of cholesteric liquid crystals. The problem treated here is the Cauchy problem for the same wave speed case
with initial data on the parabolic degenerating line. By introducing a partial hodograph transformation,
we establish the local existence of smooth solutions in a weighted metric space based on the iteration method. A classical solution of the primary problem is constructed by converting the solution in the partial hodograph variables to that in the original variables.
\end{abstract}

\begin{keywords}
Variational wave system, degenerate hyperbolic, Cauchy problem, classical solution,
weighted metric space
\end{keywords}

\begin{AMS}
35L20, 35L70, 35L80
\end{AMS}

\section{Introduction}\label{1.1}

We are interested in the degenerate Cauchy problem for the one-dimensional nonlinear system of variational wave equations
\begin{align}\label{a1}
\left\{
\begin{array}{l}
u_{tt}-(c_1^2(u)u_x)_x= -c_1(u)c_1'(u)u^2_x+a(u)a'(u)[v^2_t-c_{2}^2(u)v^2_x]\\
\qquad \qquad\qquad\qquad\quad -a^2(u)c_2(u)c_2'(u)v^2_x+2\lambda a(u)a'(u)v_x,\\
  (a^2(u)v_t)_t-[a^2(u)c_{2}^2(u)v_x-\lambda a^2(u)]_x=0,
\end{array}
\right.
\end{align}
where $t$-$x$ are the time-space independent variables, $(u,v)$ are the dependent variables, $c_1$, $c_2$ and $a$ are smooth functions of $u$, the prime means the derivative with respect to $u$, and $\lambda$ is a constant.

System \eqref{a1} is derived from the theory of chiral nematic liquid crystals or cholesteric liquid crystals. In cholesteric liquid crystals, the average orientation of the long molecules can be described by a director field $\textbf{n}\in\mathbb{S}^2$. Associated with the director field $\textbf{n}$, the well-known Frank-Oseen potential energy density $W$ is expressed as the sum of the elastic and the chiral contribution (neglecting a constant
factor)
\begin{align}\label{a2}
W(\textbf{n},\nabla \textbf{n})=\bigg\{\fr{1}{2}k_1(\nabla\cdot\textbf{n})^2 +\fr{1}{2}k_2(\textbf{n}\cdot\nabla\times\textbf{n})^2 +\fr{1}{2}k_3|\textbf{n}\times(\nabla\times\textbf{n})|^2 \bigg\} +\lambda\textbf{n}\cdot\nabla\times\textbf{n},
\end{align}
where $k_1, k_2$ and $k_3$ are the splay, twist and bend elastic
constants of the material, respectively, see e.g. \cite{Frank, Leslie, Stewart}. The constant $\lambda$ is a material parameter representing molecular chirality given by $\lambda=\pm2\pi k_2/p_0$, where $p_0$ is the pitch of the cholesteric helix and the
sign depends on the handedness of the cholesteric liquid crystal. Note that a classical nematic liquid crystal
is actually a cholesteric with infinite pitch. For detailed information regarding cholesteric
liquid crystals, see, for example, \cite{Collings, Gennes, Oswald, Stephen}. In the regime in which inertia effects dominate viscosity, the propagation of the orientation waves in the director field then can be modeled by the least action principle \cite{Saxton1989, Ali0}
\begin{align}\label{a3}
\delta\int \bigg(\fr{1}{2}\pa_t\textbf{n}\cdot\pa_t\textbf{n}-W_n(\textbf{n},\nabla \textbf{n})\bigg){\rm d}\textbf{x}{\rm d}t=0,\quad \textbf{n}\cdot\textbf{n}=1.
\end{align}
For planar deformations depending on a single space variable $x$, that is, the director field has the special form $\textbf{n}=(\cos u(x,t), \sin u(x,t), 0)$ where $u$ measures the angle of the director field to
the $x$-direction, the Euler-Lagrange equation of the variational principle \eqref{a3} reads that
\begin{align}\label{1.6}
u_{tt}-c_1(u)(c_1(u)u_x)_x=0,
\end{align}
with $c_{1}^2(u)=k_1\sin^2u+k_3\cos^2u$, which is exactly identical to the equation corresponding
to the nematic case and has been widely studied since its introduction by Hunter and Saxton \cite{Hunter-Saxton}.
To clarify the effect of the chiral contribution in \eqref{a2}, it should be considered the three-dimensional deformations which include the twist deformations. Taking the director field $\textbf{n}$ as the following form
$$
\textbf{n}=(\cos u, \sin u\cos v, \sin u\sin v),
$$
where $u$ and $v$ are spherical polar angles and are functions of $(x,t)$, the Lagrangian density of \eqref{a3} is
\begin{align}\label{a4}
\fr{1}{2}\pa_t\textbf{n}\cdot\pa_t\textbf{n}-W(\textbf{n},\nabla \textbf{n})=\fr{1}{2}[u_{t}^2-c_{1}^2(u)u_{x}^2]+ \fr{1}{2}a^2(u)[v_{t}^2-c_{2}^2(u)v_{x}^2] +\lambda a^2(u)v_x,
\end{align}
where
$$
c_{1}^2(u)=k_1\sin^2u+k_3\cos^2u,\quad c_{2}^2(u)=k_2\sin^2u+k_3\cos^2u, \quad a^2(u)=\sin^2u.
$$
The system of Euler-Lagrange equations gives \eqref{a1}, see Hu \cite{Hu2017} for more details on the derivation. When $\lambda=0$, i.e., the pitch of the cholesteric $p_0=\infty$, system \eqref{a1} reduces to
\begin{align}\label{a5}
\left\{
  \begin{array}{lll}
    u_{tt}-(c_{1}^2(u)u_x)_x =-c_1(u)\pa_uc_{1}(u)u_{x}^2 +a(u)\pa_ua(u)[v_{t}^2-c_{2}^2(u)v_{x}^2] \\ \qquad \qquad \qquad \qquad \quad  -a^2(u)c_{2}(u,x)\pa_uc_{2}(u)v_{x}^2, \\
    (a^2(u)v_t)_t-[a^2(u)c_{2}^2(u)v_x]_x=0,
  \end{array}
\right.
\end{align}
which was first derived by Ali and Hunter \cite{Ali1} from the theory of nematic liquid crystals. We point out that the presence of parameter $\lambda$ in \eqref{a1} has a great influence on the results, see \cite{Hu2017} and this paper below.

Many efforts have been made to study the Cauchy problem for the variational wave equations \eqref{1.6} and \eqref{a5} under the assumptions that the wave speeds $c_1(\cdot)$ and $c_2(\cdot)$ are both positive functions. For the variational wave equation \eqref{1.6}, the formation of cusp-type singularities was shown in \cite{Glassey}, the existence of dissipative weak solutions was investigated in \cite{Bres-Huang, Zha-Zhe2001, Zha-Zhe2003, Zha-Zhe2005}, the existence and uniqueness of conservative weak solutions was provided in \cite{B, B-C-Z, Bres-Zheng, Holden}, the stability and the generic regularity of conservative weak solutions were presented in \cite{B-C1, B-C2}. A more general variational wave equation than \eqref{1.6} was explored in \cite{Hu2015}.
For the system of variational wave equations \eqref{a5}, Zhang and Zheng established the global existence of conservative weak solutions in \cite{Zha-Zhe2010} for the case $c_1=c_2$ and in \cite{Zha-Zhe2012} for the case $c_1<c_2$. An assumption on the function $a(\cdot)$ was got rid of in \cite{Chen-Zha-Zhe} by considering the director field $\textbf{n}$ in its natural three-component form. The readers are referred to \cite{Cai, Hu2012, Hu2017} for more discussions on the systems of variational wave equations.

In general, the elastic constants are positive and then the wave speeds $c_1(\cdot)$ and $c_2(\cdot)$ are strictly positive functions. However, in some cases, see e.g. \cite{Adlem, Dozov, Panov}, the elastic constants may be negative which implies that the wave speeds $c_1(\cdot)$ and $c_2(\cdot)$ can be zero. In \cite{Saxton1992}, Saxton examined the blow-up properties of smooth solutions to the degenerate hyperbolic equation \eqref{1.6} by setting one elastic constant to zero.
For another important application, if $c_1(u)=u$, then \eqref{1.6} corresponds to the second sound equation
\begin{align}\label{a6}
u_{tt}-u(uu_x)_x=0,
\end{align}
introduced by Kato and Sugiyama \cite{Kato}. The local existence of the Cauchy problem for \eqref{a6} was established in \cite{Kato} under the assumption $u(0,x)\geq A>0$. In \cite{Hu-Wang2}, Hu and Wang studied
the local existence of classical solutions to the Cauchy problem of \eqref{1.6} with initial
data given on the parabolic degenerating line. They \cite{Hu-Wang1} also discussed the
global existence of smooth solutions to a degenerate initial-boundary value problem under relaxed data.
The studies on the degenerate hyperbolic problems for the nonlinear variational wave equations
are still very limited so far. We also refer the reader to Refs. \cite{Hu-Li, Zhang-Zheng1, Zhang-Zheng2} for the works of the sonic-supersonic structures to the compressible Euler equations in gas dynamics.

We are concerned with the local existence of classical solutions to the Cauchy problem for the nonlinear variational wave system \eqref{a1} with degenerate initial data. The results are divided into two parts.
The current paper is the first part, dealing with system \eqref{a1} (and system \eqref{a5}) with $c_1=c_2$.
The case $c_1\neq c_2$ will be handled in the next paper \cite{Hu-Song}. Comparing to the strictly hyperbolic case, the main difficulty here is to treat the singularity caused by the hyperbolic degeneracy. To overcome this
difficulty, a partial hodograph transformation is introduced to transform the equations into a new system with a clear singularity-regularity structure. With a choice of weighted metric space, the local existence of classical solutions for the new system is established by employing the fixed-point method. It will be seen that the parameter $\lambda$ greatly affects the convergence of iterative sequence generated by the integral system. Moreover, we point out that the existence of framework presented here is difficult in taking the case $c_1\neq c_2$ due to the coupling of different characteristic fields, see the next paper \cite{Hu-Song} for details.

The rest of the paper is organized as follows. In Section 2, we formulate the degenerate Cauchy problem and then state the main results of the paper. Section 3 is devoted to transforming the problem into a new problem under a partial hodograph plane. Finally, we solve the new problem in a weighted metric space and then complete the proof of the main results.

\section{The problem and the main results}

Set $c_1=c_2=c$. Then system \eqref{a1} reduces to
\begin{align}\label{1.1}
\left\{
\begin{array}{l}
u_{tt}-(c^2(u)u_x)_x= -c(u)c'(u)u^2_x+a(u)a'(u)[v^2_t-c^2(u)v^2_x]\\
\qquad \qquad\qquad\qquad\quad -a^2(u)c(u)c'(u)v^2_x+2\lambda a(u)a'(u)v_x,\\
  (a^2(u)v_t)_t-[a^2(u)c^2(u)v_x-\lambda a^2(u)]_x=0.
\end{array}
\right.
\end{align}
We assume that the functions $c(\cdot)$ and $a(\cdot)$ satisfy
\begin{align}\label{1.8}
|c^{(l)}(z)|; |a^{(l)}(z)|<\infty\ (l=1,2,3),\quad |c'(z)|;|a(z)|\geq m_0>0, \quad \forall\ z\in R,
\end{align}
for some positive constant $m_0$, and then consider the Cauchy problem to \eqref{1.1} with the following initial data
\begin{align}\label{1.10}
\left\{
\begin{array}{l}
u(0,x)=\varphi_1, \quad\ \  u_{t}(0,x)=\psi_1(x), \\
v(0,x)=\varphi_2(x),\ v_{t}(0,x)=\psi_2(x),  \\
c(\varphi_1)=0.
\end{array}
\right.
\begin{array}{l}
\ \ \forall\ x\in R,
\end{array}
\end{align}
where $\varphi_1$ is a constant, $\varphi_2$, $\psi_1$ and $\psi_2$ are smooth functions. It is noted that
the wave speed  $c$ is zero on the initial line, which means the hyperbolic system \eqref{1.1} is parabolic degenerate at $t=0$.

Denote
\begin{align}\label{1.11}
\left\{
\begin{array}{l}
R_1=u_t+c(u)u_x,\\
S_1=u_t-c(u)u_x,
\end{array}
\right.\quad
\left\{
\begin{array}{l}
R_2=v_t+c(u)v_x,\\
S_2=v_t-c(u)v_x,
\end{array}
\right.
\end{align}
so that
\begin{align}\label{1.12a}
u_t=\fr{R_1+S_1}{2},\quad u_x=\fr{R_1-S_1}{2c},\quad v_t=\fr{R_2+S_2}{2},\quad v_x=\fr{R_2-S_2}{2c}.
\end{align}
Then, by \eqref{1.1}, we can obtain a first-order system in terms of $(R_1, S_1, R_2, S_2, u, v)$
\begin{align}\label{1.12}
\left\{
\begin{array}{l}
\dps R_{1t}-cR_{1x}=\fr{c'(R_1+S_1)}{4}\fr{R_1-S_1}{c}+\lambda aa'\fr{R_2-S_2}{c}-\fr{a^2c'}{4}\fr{(R_2-S_2)^2}{c}+aa'R_2S_2,\\[6pt]
\dps S_{1t}+cS_{1x}=\fr{c'(R_1+S_1)}{4}\fr{S_1-R_1}{c}+\lambda aa'\fr{R_2-S_2}{c}-\fr{a^2c'}{4}\fr{(R_2-S_2)^2}{c}+aa'R_2S_2,\\[6pt]
\dps R_{2t}-cR_{2x}=\fr{c'R_1}{2}\fr{R_2-S_2}{c}-\fr{\lambda a'}{a}\fr{R_1-S_1}{c}-\fr{a'}{a}(R_1S_2+S_1R_2),\\[6pt]
\dps S_{2t}+cS_{2x}=\fr{c'S_1}{2}\fr{S_2-R_2}{c}-\fr{\lambda a'}{a}\fr{R_1-S_1}{c}-\fr{a'}{a}(R_1S_2+S_1R_2),\\[6pt]
\dps u_t=\fr{R_1+S_1}{2},\quad v_t=\fr{R_2+S_2}{2}.
\end{array}
\right.
\end{align}
We look for classical solutions to system \eqref{1.12} with the following initial conditions
\begin{align}\label{1.13}
\begin{array}{l}
\dps R_1(0,x)=\psi_1(x),\quad S_1(0,x)=\psi_1(x),\quad R_2(0,x)=\psi_2(x),\quad S_2(0,x)=\psi_2(x),\\
\dps R_{1t}(0,x)=S_{1t}(0,x)=a(\varphi_1)a'(\varphi_1)\psi^2_{2}(x)+2\lambda a(\varphi_1)a'(\varphi_1)\varphi'_2(x):=f_{11}(x),\\[3pt]
\dps R_{2t}(0,x)=c'(\varphi_1)\psi_1(x)\varphi'_2(x) -\fr{2a'(\varphi_1)}{a(\varphi_1)}\psi_1(x)\psi_2(x):=f_{21}(x),\\[8pt]
\dps S_{2t}(0,x)=-c'(\varphi_1)\psi_1(x)\varphi'_2(x) -\fr{2a'(\varphi_1)}{a(\varphi_1)}\psi_1(x)\psi_2(x):=f_{22}(x),\\[3pt]
\dps u(0,x)=\varphi_1,\quad u_t(0,x)=\psi_1(x),\quad v(0,x)=\varphi_2(x),\quad v_t(0,x)=\psi_2(x).
\end{array}
\end{align}

It is worthwhile mentioning that the local existence of the degenerate Cauchy problem \eqref{1.12} \eqref{1.13} cannot be solved by the classical local existence theory of nonlinear hyperbolic equations in \cite{LiT, Wang}. The reason is that system \eqref{1.12} is not a continuously differentiable system by the degeneracy. We isolate the singularities of the system in a partial hodograph plane and then establish the existence of solutions in a weighted metric space. Finally, by expressing in terms of $(x,t)$ plane, we obtain the classical solutions of problem \eqref{1.12} \eqref{1.13} and so of problem \eqref{1.1} \eqref{1.10}. The main conclusions of this paper can be stated as follows.
\begin{thm}\label{thm1}
Suppose that \eqref{1.8} holds and functions $\varphi_2$, $\psi_1$ and $\psi_2$ satisfy
\begin{align}\label{1.14}
\begin{array}{c}
|\varphi^{(j)}_2(x)|<\infty\ (j=1,\cdots,4),\quad |\psi^{(k)}_1(x)|;|\psi^{(k)}_2(x)|<\infty \ (k=1,2,3), \\ |\psi_1(x)|\geq\psi_0>0,
\end{array}
\end{align}
for all $x\in R$ and some constant $\psi_0$. Then there exist constants $\lambda_0>0$ and $\delta>0$ such that the degenerate Cauchy problem \eqref{1.12} \eqref{1.13} with $|\lambda|\leq\lambda_0$ has a classical solution on $[0,\delta]\times R$.
\end{thm}

From Theorem \ref{thm1}, we directly have
\begin{thm}\label{thm2}
Let the assumptions in Theorem \ref{thm1} hold. Then there exist constants $\lambda_0>0$ and $\delta>0$ such that the degenerate Cauchy problem \eqref{1.1} \eqref{1.10} with $|\lambda|\leq\lambda_0$ has a classical solution on $[0,\delta]\times R$.
\end{thm}

\section{The reformulation of the problem}

We just only deal with the case $c'(u)\leq-m_0$ and $\psi_1(x)\geq\psi_0$, the other cases can be discussed analogously. Introduce the partial hodograph transformation $(x,t)\rightarrow(y,\tau)$ by defining
\begin{align}\label{2.1}
\tau=-c(u(t,x)),\quad y=x.
\end{align}
Thanks to \eqref{1.12a}, the jacobian of this transformation is
\begin{align}\label{2.2}
J:=\fr{\partial(y, \tau)}{\partial(x, t)}=y_x\tau_t-y_t\tau_x=-c_uu_t=-c'\fr{R_1+S_1}{2},
\end{align}
which is strictly positive at $t=0$ by the assumptions. Furthermore, we have
\begin{align}\label{2.3}
\partial_t=-\fr{c'(R_1+S_1)}{2}\partial_\tau,\quad \partial_x=\partial_y+\fr{c'(R_1-S_1)}{2\tau}\partial_\tau.
\end{align}
In terms of the new coordinates $(y, \tau)$, system \eqref{1.12} can be rewritten as
\begin{align}\label{2.4}
\left\{
\begin{array}{l}
\dps   R_{1\tau}-\fr{\tau}{c'S_1}R_{1y}=\fr{R_1+S_1}{4S_1}\fr{R_1-S_1}{\tau}+\fr{\lambda aa'}{c'S_1}\fr{R_2-S_2}{\tau}
-\fr{a^2}{4S_1}\fr{(R_2-S_2)^2}{\tau}-\fr{aa'}{c'S_1}R_2S_2,\\[6pt]
\dps  S_{1\tau}+\fr{\tau}{c'R_1}S_{1y}=\fr{R_1+S_1}{4R_1}\fr{S_1-R_1}{\tau}+\fr{\lambda aa'}{c'R_1}\fr{R_2-S_2}{\tau}
-\fr{a^2}{4R_1}\fr{(R_2-S_2)^2}{\tau}-\fr{aa'}{c'R_1}R_2S_2,\\[6pt]
\dps  R_{2\tau}-\fr{\tau}{c'S_1}R_{2y}=\fr{R_1}{2S_1}\fr{R_2-S_2}{\tau}-\fr{\lambda a'}{ac'S_1}\fr{R_1-S_1}{\tau} +\fr{a'}{ac'S_1}(R_1S_2+R_2S_1),\\[6pt]
\dps  S_{2\tau}+\fr{\tau}{c'R_1}S_{2y}=\fr{S_1}{2R_1}\fr{S_2-R_2}{\tau}-\fr{\lambda a'}{ac'R_1}\fr{R_1-S_1}{\tau} +\fr{a'}{ac'R_1}(R_1S_2+R_2S_1),
\end{array}
\right.
\end{align}
with two decoupled equations
\begin{align}\label{2.5}
u_\tau=-\fr{1}{c'(u)},\quad v_\tau=-\fr{R_2+S_2}{c'(R_1+S_1)}.
\end{align}
Note that the equation for $u$ in \eqref{2.5} is a trivial equation and the equation for $v$ is not needed because the coefficients in system \eqref{2.4} are independent of $v$. Corresponding to \eqref{1.13}, one can easy to get the initial conditions of system \eqref{2.4} in the coordinates $(y, \tau)$
\begin{align}\label{2.7}
\begin{array}{c}
R_1(0,y)=S_1(0,y)=\psi_1(y),\quad  R_2(0,y)=S_2(0,y)=\psi_2(y), \\
R_{1\tau}(0,y)=S_{1\tau}(0,y)=g_{11}(y), \quad R_{2\tau}(0,y)=g_{21}(y),\quad S_{2\tau}(0,y)=g_{22}(y),
\end{array}
\end{align}
where
$$
g_{11}(y)=-\fr{f_{11}(y)}{c'(\varphi_1)\psi_1(y)},\quad g_{21}(y)=-\fr{f_{21}(y)}{c'(\varphi_1)\psi_1(y)},\quad g_{22}(y)=-\fr{f_{22}(y)}{c'(\varphi_1)\psi_1(y)}.
$$

We now homogenize the boundary conditions \eqref{2.7} of system \eqref{2.4} by introducing the new dependent variables as follows
\begin{align}\label{2.8}
\begin{array}{l}
U_1(\tau,y)=R_1(\tau,y)-\psi_1(y)-g_{11}(y)\tau,\qquad U_2(\tau,y)=S_1(\tau,y)-\psi_1(y)-g_{11}(y)\tau,\\
U_3(\tau,y)=R_2(\tau,y)-\psi_2(y)-g_{21}(y)\tau,\qquad U_4(\tau,y)=S_2(\tau,y)-\psi_2(y)-g_{22}(y)\tau,
\end{array}
\end{align}
from which,one has
\begin{align*}
\begin{array}{l}
R_1=U_1+\psi_1+g_{11}\tau,\quad
S_1=U_2+\psi_1+g_{11}\tau,
\\
R_2=U_3+\psi_2+g_{21}\tau,\quad
S_2=U_4+\psi_2+g_{22}\tau,
\end{array}
\end{align*}
and
\begin{align*}
\begin{array}{l}
R_1+S_1=U_1+U_2+2\psi_1+2g_{11}\tau,\qquad \qquad
R_1-S_1=U_1-U_2,\\
R_2+S_2=U_3+U_4+2\psi_2+(g_{21}+g_{22})\tau,\quad
R_2-S_2=U_3-U_4-2\varphi'_2\tau.
\end{array}
\end{align*}
It follows by \eqref{2.7} and \eqref{2.8} that
\begin{align}\label{2.9}
\begin{array}{l}
U_i(0,y)=U_{i\tau}(0,y)=0,\ (i=1,2,3,4).
\end{array}
\end{align}
By performing a direct calculation, we obtain the equations for $\textbf{U}=(U_1,U_2,U_3,U_4)^T$
\begin{align}\label{2.10}
\left\{
\begin{array}{l}
\dps  U_{1\tau}-\fr{\tau}{c'(U_2+g)}U_{1y}=\fr{U_1-U_2}{2\tau}+\fr{\lambda aa'}{c'(U_2+g)}\fr{U_3-U_4}{\tau}+T_1(\tau,y,\textbf{U}),\\[8pt]
\dps U_{2\tau}+\fr{\tau}{c'(U_1+g)}U_{2y}=\fr{U_2-U_1}{2\tau}+\fr{\lambda aa'}{c'(U_1+g)}\fr{U_3-U_4}{\tau}+T_2(\tau,y,\textbf{U}),\\[8pt]
\dps U_{3\tau}-\fr{\tau}{c'(U_2+g)}U_{3y}=\fr{U_3-U_4}{2\tau}-\fr{\lambda a'}{ac'(U_2+g)}\fr{U_1-U_2}{\tau}+T_3(\tau,y,\textbf{U}),\\[8pt]
\dps U_{4\tau}+\fr{\tau}{c'(U_1+g)}U_{4y}=\fr{U_4-U_3}{2\tau}-\fr{\lambda a'}{ac'(U_1+g)}\fr{U_1-U_2}{\tau}+T_4(\tau,y,\textbf{U}),
\end{array}
\right.
\end{align}
where $g=\psi_1+g_{11}\tau$ and
\begin{align}\label{2.11}
\begin{array}{l}
\dps T_1(\tau,y,\textbf{U})=\fr{1}{4(U_2+g)}\fr{(U_1-U_2)^2}{\tau}-\fr{a^2}{4(U_2+g)}\fr{(U_3-U_4)^2}{\tau} +\dps\sum_{j=1}^4T_{1j}U_j+F_1\tau,\\
\dps T_2(\tau,y,\textbf{U})=\fr{1}{4(U_1+g)}\fr{(U_2-U_1)^2}{\tau}-\fr{a^2}{4(U_1+g)}\fr{(U_3-U_4)^2}{\tau} +\dps\sum_{j=1}^4T_{2j}U_j+F_2\tau,\\
\dps T_3(\tau,y,\textbf{U})=\fr{U_1-U_2}{2(U_2+g)}\cdot\fr{U_3-U_4}{\tau}+\dps\sum_{j=1}^4T_{3j}U_j+F_3\tau,\\
\dps T_4(\tau,y,\textbf{U})=\fr{U_2-U_1}{2(U_1+g)}\cdot\fr{U_4-U_3}{\tau}+\dps\sum_{j=1}^4T_{4j}U_j+F_4\tau.
\end{array}
\end{align}
Here the detailed expressions of $T_{ij}$ and $F_i\ (i,j=1,2,3,4)$ are given in Appendix \ref{app}.

Let $\mathcal{F}=\mathcal{F}(M,\delta)$ be a function class consisting of all continuously differentiable function $\textbf{u}=(u_1,u_2,u_3,u_4)^T$: $[0,\delta]\times R \rightarrow R^2$ satisfying the following properties:\\
$(P_1)\quad u_j(0,y)=u_{j\tau}(0,y)=0$, \quad $(j=1,2,3,4)$,\\
$(P_2)\quad \dps\sum_{j=1}^4\bigg\|\fr{u_j}{\tau^2}\bigg\|_{L^\infty}\leq M$,\\
$(P_3)\quad \dps\sum_{j=1}^4\bigg\|\fr{u_{jy}}{\tau^2}\bigg\|_{L^\infty}\leq M$,\\
$(P_4)\quad u_{jy}\ (j=1,2,3,4)$ are Lipschitz continuous with respect to $y$ with $\dps\sum_{j=1}^4\bigg\|\fr{u_{jyy}}{\tau^2}\bigg\|_{L^\infty}\leq M$,\\
where $M$ and $\delta$ are two positive constants. We use $\mathcal{H}$ to denote the function class containing only continuous functions on $[0,\delta]\times R$ satisfying only $(P_1)$ and $(P_2)$. It is easily known that
$\mathcal{F}$ is a subset of $\mathcal{H}$ and both of them are subsets of $C^0([0,\delta]\times R; R^2)$. Let $\textbf{u}$ and $\hat{\textbf{u}}$ be any two elements in $\mathcal{H}$. We define a weighted metric on $\mathcal{F}$ and $\mathcal{H}$
\begin{align}\label{2.12}
\begin{array}{l}
d(\textbf{u},\hat{\textbf{u}}):=\dps\sum_{j=1}^4\bigg\|\fr{u_j-\hat{u}_j}{\tau^2}\bigg\|_{L^\infty}.
\end{array}
\end{align}
One can check that $(\mathcal{H},d)$ is a complete metric space, while $(\mathcal{F},d)$ is not a closed subset in $(\mathcal{H},d)$.

Our strategy is to show first the existence of classical solutions for the homogeneous initial value problem
\eqref{2.10} \eqref{2.9}. Then we obtain the main results Theorems \ref{thm1} and \ref{thm2} by the coordinate change $(y, \tau)\rightarrow(x, t)$. For the problem \eqref{2.10} \eqref{2.9}, we have

\begin{thm}\label{thm3}
Assume that the conditions listed in Theorem \ref{thm1} hold. Then there exist constants $\lambda_0>0$ and $\delta>0$ such that the degenerate hyperbolic problem \eqref{2.10} \eqref{2.9} with $|\lambda|\leq\lambda_0$ has a classical solution on $[0,\delta]\times R$ in the function class $\mathcal{F}$.
\end{thm}

\section{The proof of the main results}

In this section, we establish Theorem \ref{thm3} by the fixed point method and then complete the proof of Theorems \ref{thm1} and \ref{thm2} by converting the solution in terms of the variables $(x,t)$. The process is divided into several steps.

\textbf{Step 1: The iteration mapping.} Denote
\begin{align}\label{3.1}
\Lambda_+(U_2)=-\fr{\tau}{c'(U_2+g)},\quad \Lambda_-(U_1)=\fr{\tau}{c'(U_1+g)},
\end{align}
and
\begin{align}\label{3.2}
\fr{{\rm d}}{{\rm d}_+(U_2)}=\partial_\tau+\Lambda_+(U_2)\partial_y,\quad \fr{{\rm d}}{{\rm d}_-(U_1)}=\partial_\tau+\Lambda_-(U_1)\partial_y.
\end{align}
Let $\textbf{u}=(u_1,u_2,u_3,u_4)^T(\tau,y)$ be an element in the set $\mathcal{F}$. We consider the linear system
\begin{align}\label{3.3}
\left\{
\begin{array}{l}
\dps \fr{\rm d}{{\rm d}_+(u_2)}U_1=\fr{u_1-u_2}{2\tau}+\fr{\lambda a'a}{c'(u_2+g)}\fr{u_3-u_4}{\tau}+T_1(\tau,y,\textbf{u}),\\[8pt]
\dps \fr{{\rm d}}{{\rm d}_-(u_1)}U_2=\fr{u_2-u_1}{2\tau}+\fr{\lambda a'a}{c'(u_1+g)}\fr{u_3-u_4}{\tau}+T_2(\tau,y,\textbf{u}),\\[8pt]
\dps \fr{{\rm d}}{{\rm d}_+(u_2)}U_3=\fr{u_3-u_4}{2\tau}-\fr{\lambda a'}{ac'(u_2+g)}\fr{u_1-u_2}{\tau}+T_3(\tau,y,\textbf{u}),\\[8pt]
\dps \fr{{\rm d}}{{\rm d}_-(u_1)}U_4=\fr{u_4-u_3}{2\tau}-\fr{\lambda
a'}{ac'(u_1+g)}\fr{u_1-u_2}{\tau}+T_4(\tau,y,\textbf{u}).
\end{array}
\right.
\end{align}
From \eqref{3.3} and \eqref{2.9}, we acquire
\begin{align}\label{3.4}
\left\{
\begin{array}{l}
U_1(\xi,\eta)=\dps\int_{0}^\xi\bigg\{\fr{u_1-u_2}{2\tau}+\fr{\lambda a'a}{c'(u_2+g)}\fr{u_3-u_4}{\tau}+T_1\bigg\}(\tau,y_+(\tau;\xi,\eta)){\rm d}\tau,\\[8pt]
U_2(\xi,\eta)=\dps\int_{0}^\xi\bigg\{\fr{u_2-u_1}{2\tau}+\fr{\lambda a'a}{c'(u_1+g)}\fr{u_3-u_4}{\tau}+T_2\bigg\}(\tau,y_-(\tau;\xi,\eta)){\rm d}\tau,\\[8pt]
U_3(\xi,\eta)=\dps\int_{0}^\xi\bigg\{\fr{u_3-u_4}{2\tau}-\fr{\lambda a'}{ac'(u_2+g)}\fr{u_1-u_2}{\tau}+T_3\bigg\}(\tau,y_+(\tau;\xi,\eta)){\rm d}\tau,\\[8pt]
U_4(\xi,\eta)=\dps\int_{0}^\xi\bigg\{\fr{u_4-u_3}{2\tau}-\fr{\lambda a'}{ac'(u_1+g)}\fr{u_1-u_2}{\tau}+T_4\bigg\}(\tau,y_-(\tau;\xi,\eta)){\rm d}\tau,
\end{array}
\right.
\end{align}
where $y_+(\tau;\xi,\eta)$ and $y_-(\tau;\xi,\eta)$ are defined as
\begin{align}\label{3.5}
\left\{
\begin{array}{l}
\dps\fr{{\rm d}y_+}{{\rm d}\tau}=\Lambda_+(u_2),\\
y_+(\xi;\xi,\eta)=\eta,
\end{array}
\right.
\left\{
\begin{array}{l}
\dps\fr{{\rm d}y_-}{{\rm d}\tau}=\Lambda_-(u_1),\\
y_-(\xi;\xi,\eta)=\eta,
\end{array}
\right.
\end{align}
and $T_1(\tau,y_+(\tau;\xi,\eta))=T_1(\tau,y_+(\tau;\xi,\eta),\textbf{u}(\tau,y_+(\tau;\xi,\eta)))$, etc.
Based on \eqref{3.4}, we arrive at a mapping
\begin{align*}
\mathcal{T} \left( \left(
 \begin{array}{c}
    u_1 \\
    u_2 \\
    u_3 \\
    u_4
 \end{array}
\right )\right ) \quad
=\ \ \ \left(
 \begin{array}{c}
    U_1 \\
    U_2 \\
    U_3 \\
    U_4
 \end{array}
\right).
\end{align*}
Hence the problem is changed to find a fixed point of the mapping $\mathcal{T}$ in the set $\mathcal{F}$.

\textbf{Step 2: Properties of the mapping.} Throughout the paper, we use the notation $K>1$ to denote a constant depending only on the constants $m_0$, $\psi_0$ and the $C^3$ norms of $c, a, \varphi'_2$, $\psi_1$, $\psi_2$,  which may change from one line to the next.

Thanks to $(u_1,u_2,u_3,u_4)^T\in\mathcal{F}$, we find by \eqref{1.8} and \eqref{1.14} that there exists a small constant $\delta_0>0$ such that for $\tau\leq\delta_0$
\begin{align}\label{4.1}
|c'(u_i+g)|=|c'(u_i+\psi_1+g_{11}\tau)|\geq|c'\psi_1|-|c'u_i+c'g_{11}\tau| \nonumber \\
\qquad\qquad\quad\geq m_0\psi_0-\delta_0(|c'|M\delta_0+|c'g_{12}|)\geq\fr{m_0\psi_0}{2}>0,
\end{align}
and
\begin{align}\label{4.2}
|ac'(u_i+g)|\geq|a|\cdot|c'|\cdot|u_i+g|\geq\fr{m_{0}^2\psi_0}{2}>0,
\end{align}
for $i=1,2$.

We now establish the properties of the mapping $\mathcal{T}$.
\begin{lem}\label{lem1}
Let the assumptions in Theorem \ref{thm3} hold. Then there exist positive constants $\delta\leq\delta_0$, $\lambda_0$, $M$ and $0<\kappa<1$ depending only on constants $m_0$,$\psi_0$ and the $C^3$ norms of $c$, $a$, $\varphi'_2$, $\psi_1$, $\psi_2$ such that for $|\lambda|\leq\lambda_0$ \\
$(1)\ \mathcal{T}$ map $\mathcal{F}$ into $\mathcal{F}$;\\
$(2)$ For any pair $\textbf{u}$, $\hat{\textbf{u}}$ in $\mathcal{F}$,there holds\\
\begin{align}\label{3.6}
\begin{array}{l}
d(\mathcal{T}(\textbf{u}),\mathcal{T}(\hat{\textbf{u}}))\leq \kappa d(\textbf{u},\hat{\textbf{u}}).
\end{array}
\end{align}
\end{lem}
\begin{proof}
Let $\textbf{u}=(u_1,u_2,u_3,u_4)^T$ and $\hat{\textbf{u}}=(\hat{u}_1,\hat{u}_2,\hat{u}_3,\hat{u}_4)^T$ be two elements in the set $\mathcal{F}$. We denote $\textbf{U}=\mathcal{T}(\textbf{u})=(U_1,U_2,U_3,U_4)^T$ and $\hat{\textbf{U}}=\mathcal{T}(\hat{\textbf{u}})=(\hat{U}_1,\hat{U}_2,\hat{U}_3,\hat{U}_4)$. It follows by $\textbf{u}\in\mathcal{F}$ that
\begin{align}\label{3.7}
\begin{array}{c}
|u_1-u_2|+|u_3-u_4|\leq M\tau^2, \\
|u_{1y}-u_{2y}|+|u_{3y}-u_{4y}|\leq M\tau^2, \\
|u_{1yy}-u_{2yy}|+|u_{3yy}-u_{4yy}|\leq M\tau^2.
\end{array}
\end{align}
Moreover, we set
\begin{align*}
A_1=\fr{a'a}{c'(u_2+g)},\quad A_2=\fr{a'a}{c'(u_1+g)},\quad A_3=-\fr{a'}{ac'(u_2+g)},\quad A_4=-\fr{a'}{ac'(u_1+g)},
\end{align*}
then there hold by \eqref{4.1} and \eqref{4.2}
\begin{align}\label{3.8}
|\lambda A_i|\leq \lambda K,\ (i=1,2,3,4).
\end{align}
Combining with \eqref{3.7} and \eqref{3.8}, we obtain the estimate of $T_1$ by using the detailed expressions of $T_{1i}\ (i=1,2,3,4)$ and $F_1$
\begin{align}\label{3.10}
|T_1|\leq & \bigg|\fr{c'}{4 c'(u_2+g)}\fr{(u_1-u_2)^2}{\tau}\bigg|+\bigg|\fr{a^2c'}{4c'(u_2+g)}\fr{(u_3-u_4)^2}{\tau}\bigg|
+\dps\sum^{4}_{j=1}|T_{1j}u_j|+|F_1|\tau \nonumber \\
& \leq KM^2\tau^3+KM^2\tau^3+K(1+M\delta)M\tau^2+K(1+M\delta)\tau \nonumber \\
& \leq K\tau(1+M\delta)^2.
\end{align}
Similar arguments lead to
\begin{align}\label{3.11}
|T_i|\leq K\tau(1+M\delta)^2,\ (i=1,2,3,4).
\end{align}
Summing up \eqref{3.7}, \eqref{3.8} and \eqref{3.11} gives
\begin{align}\label{3.11a}
U_i(0,\eta)=0,\ (i=1,2,3,4).
\end{align}

Furthermore, it suggests by \eqref{3.4}, \eqref{3.8} and \eqref{3.11} that
\begin{align}\label{3.11b}
&|U_1(\xi,\eta)|+|U_3(\xi,\eta)|\leq \dps\int_{0}^\xi\bigg\{\fr{|u_1-u_2|+|u_3-u_4|}{2\tau}+\bigg|\lambda A_1\fr{u_3-u_4}{\tau}\bigg| \nonumber \\
&\qquad \qquad \qquad \qquad \qquad \qquad \quad     +\bigg|\lambda A_3\fr{u_1-u_2}{\tau}\bigg|+|T_1|+|T_3|\bigg\}\ {\rm d}\tau,\nonumber \\
\leq &\dps\int_{0}^\xi\bigg\{\fr{M}{2}\tau+|\lambda| KM\tau+|\lambda| KM\tau+K\tau(1+M\delta)^2+K\tau(1+M\delta)^2\bigg\}\ {\rm d}\tau \nonumber \\
\leq & \xi^2\bigg\{\fr{M}{4}+|\lambda| KM+K(1+M\delta)^2\bigg\}.
\end{align}
In a similar way, one gets
\begin{align*}
|U_2(\xi,\eta)|+|U_4(\xi,\eta)|\leq \xi^2\bigg\{\fr{M}{4}+|\lambda| KM+K(1+M\delta)^2\bigg\},
\end{align*}
which along with \eqref{3.11b} yields
\begin{align}\label{3.12}
\dps\sum^{4}_{j=1}\bigg |\fr{U_j(\xi,\eta)}{\xi^2}\bigg |\leq M\bigg(\fr{1}{2}+|\lambda| K+\fr{K}{M}(1+M\delta)^2\bigg).
\end{align}

We now differentiate $U_1(\xi,\eta)$ and $U_3(\xi,\eta)$ with respect to $\eta$ and add the results to achieve
\begin{align}\label{3.13}
\fr{\partial U_1(\xi,\eta)}{\partial \eta} +\fr{\partial U_3(\xi,\eta)}{\partial \eta}=&\dps\int_{0}^\xi\bigg\{\fr{u_{1y}-u_{2y}}{2\tau}
+\fr{u_{3y}-u_{4y}}{2\tau}+\lambda\pa_y\bigg(A_1\fr{u_3-u_4}{\tau}\bigg) \nonumber \\
&\qquad \quad +\lambda\pa_y\bigg(A_3\fr{u_1-u_2}{\tau}\bigg)+T_{1y}+T_{3y}\bigg\}\fr{\partial y_+}{\partial \eta}\ {\rm d}\tau,
\end{align}
where
\begin{align*}
\fr{\partial y_+}{\partial \eta}(\tau;\xi,\eta)=\exp\bigg(\dps\int_{\xi}^\tau \fr{\partial \Lambda_+(u_2)}{\partial y}(s,y_+(s;\xi,\eta))\ {\rm d}s\bigg).
\end{align*}
Next we derive a series of estimates
\begin{align}\label{3.14}
|\partial_y(c'u_2+c'g)|=\big|c'u_{2y}+c''u_yu_2+c''u_yg+c'g_y\big|\leq K(1+M\delta),
\end{align}
\begin{align}\label{3.15}
&\bigg|\partial_y\bigg(A_1\fr{u_3-u_4}{\tau}\bigg)\bigg| =\bigg|\partial_y\bigg(\fr{a'a}{c'u_2+c'g}\fr{u_3-u_4}{\tau}\bigg)\bigg| \nonumber \\[3pt]
\leq &\bigg|\fr{a'a}{c'u_2+c'g}\cdot\fr{u_{3y}-u_{4y}}{\tau}\bigg|
+\bigg|\fr{(a'a)_yc'(u_2+g)-a'a(c'u_2+c'g)_y}{(c'u_2+c'g)^2}\bigg|\cdot\bigg|\fr{u_3-u_4}{\tau}\bigg| \nonumber\\[3pt]
\leq &KM\tau+(K+K(1+M\delta))M\tau  \leq K\tau(1+M\delta)M,
\end{align}
and
\begin{align}\label{3.16}
\bigg|\partial_y\bigg(A_3\fr{u_1-u_2}{\tau}\bigg)\bigg|\leq K\tau(1+M\delta)M.
\end{align}
Moreover, we take a direct calculation
\begin{align}\label{3.17}
|T_{1y}|\leq &\bigg|\partial_y\bigg(\fr{1}{4(u_2+g)}\fr{(u_1-u_2)^2}{\tau}\bigg)\bigg| +\bigg|\partial_y\bigg(-\fr{a^2}{4(u_2+g)}\fr{(u_3-u_4)^2}{\tau}\bigg)\bigg| \nonumber \\
&+\dps\sum^{4}_{j=1}|T_{1jy}u_j|+\dps\sum^{4}_{j=1}|T_{1j}u_{jy}|+|F_{1y}|\tau \nonumber \\
\leq &KM^2\tau^3+K(1+M\delta)M^2\tau^3+KM^2\tau^3+K(1+M\delta)M^2\tau^3 \nonumber \\
&+K(1+M\delta)^2M\tau^2
+K(1+M\delta)M\tau^2+K(1+M\delta)^2\tau \nonumber \\
\leq & K\tau(1+M\delta)^3,
\end{align}
and similarly
\begin{align}\label{3.18}
|T_{3y}|\leq K\tau(1+M\delta)^3.
\end{align}
According to \eqref{3.14}, one has
\begin{align}\label{3.19}
|\partial_y\Lambda_+(u_2)|=\bigg|\fr{\partial_y(c'u_2+c'g)}{(c'u_2+c'g)^2}\tau\bigg|\leq K\tau(1+M\delta),
\end{align}
and then
\begin{align}\label{3.20}
\bigg|\fr{\partial y_+}{\partial \eta}\bigg|\leq \exp\bigg(\dps\int_{0}^\xi Ks(1+M\delta)\ {\rm d}s\bigg)\leq e^{K\delta^2(1+M\delta)}.
\end{align}
Combining\eqref{3.15}-\eqref{3.20} and applying \eqref{3.7}, we get
\begin{align}\label{3.21}
\bigg|\fr{\partial U_1}{\partial \eta}\bigg|+\bigg|\fr{\partial U_3}{\partial \eta}\bigg|\leq&\dps\int_{0}^\xi\bigg\{\fr{|u_{1y}-u_{2y}|+|u_{3y}-u_{4y}|}{2\tau} +|\lambda|\cdot\bigg|\pa_y\bigg(A_1\fr{u_3-u_4}{\tau}\bigg)\bigg| \nonumber \\
&\qquad \ \ +|\lambda|\cdot\bigg|\pa_y\bigg(A_3\fr{u_1-u_2}{\tau}\bigg)\bigg|
+|T_{1y}|+|T_{3y}|\bigg\}\bigg|\fr{\partial y_+}{\partial\eta}\bigg|\ {\rm d}\tau \nonumber \\
\leq & \dps\int_{0}^\xi\bigg\{\fr{1}{2}M\tau+|\lambda| K\tau(1+M\delta)M+|\lambda| K\tau(1+M\delta)M   \nonumber\\
&\qquad \ \ +K\tau(1+M\delta)^3+K\tau(1+M\delta)^3\bigg\}e^{K\delta^2(1+M\delta)}\ {\rm d}\tau \nonumber\\
\leq &\bigg\{\fr{1}{4}M\xi^2+|\lambda| K\xi^2(1+M\delta)M+K\xi^2(1+M\delta)^3\bigg\}e^{K\delta^2(1+M\delta)} \nonumber\\
\leq &\xi^2\bigg(\fr{1}{4}M+|\lambda| K(1+M\delta)M+K(1+M\delta)^3\bigg)e^{K\delta^2(1+M\delta)}.
\end{align}
Similar arguments for $U_2$ and $U_4$ yield
\begin{align*}
\bigg|\fr{\partial U_2}{\partial \eta}\bigg|+\bigg|\fr{\partial U_4}{\partial \eta}\bigg|\leq \xi^2\bigg(\fr{1}{4}M+|\lambda| K(1+M\delta)M+K(1+M\delta)^3\bigg)e^{K\delta^2(1+M\delta)},
\end{align*}
Which along with \eqref{3.21} gives
\begin{align}\label{3.22}
\dps\sum^{4}_{j=1}\bigg |\fr{U_{j\eta}(\xi,\eta)}{\xi^2}\bigg |\leq M\bigg(\fr{1}{2}+|\lambda| K(1+M\delta)+\fr{K}{M}(1+M\delta)^3\bigg)e^{K\delta^2(1+M\delta)}.
\end{align}

To establish the bounds of $\pa_{\eta\eta}U_{1}/\xi^2$ and $\pa_{\eta\eta}U_{3}/\xi^2$, we differentiate \eqref{3.13} with respect to $\eta$ again to get
\begin{align}\label{3.23}
\fr{\partial^2 U_1(\xi,\eta)}{\partial \eta^2} +\fr{\partial^2 U_3(\xi,\eta)}{\partial \eta^2}=&\dps\int_{0}^\xi\bigg\{I_1\bigg(\fr{\partial y_+}{\partial \eta}\bigg)^2+I_2\fr{\partial^2 y_+}{\partial \eta^2}\bigg\}\ {\rm d}\tau,
\end{align}
where
\begin{align*}
I_1=&\fr{u_{1yy}-u_{2yy}}{2\tau}
+\fr{u_{3yy}-u_{4yy}}{2\tau}+\lambda\fr{\partial^2}{\partial y^2}\bigg(A_1\fr{u_3-u_4}{\tau}\bigg) \nonumber \\
&+\lambda\fr{\partial^2}{\partial y^2}\bigg(A_3\fr{u_1-u_2}{\tau}\bigg)+T_{1yy}+T_{3yy}, \nonumber \\
I_2=&\fr{u_{1y}-u_{2y}}{2\tau}
+\fr{u_{3y}-u_{4y}}{2\tau}+\lambda\fr{\partial}{\partial y}\bigg(A_1\fr{u_3-u_4}{\tau}\bigg) \nonumber \\
&+\lambda\fr{\partial}{\partial y}\bigg(A_3\fr{u_1-u_2}{\tau}\bigg)+T_{1y}+T_{3y},
\end{align*}
and
\begin{align}\label{3.23a}
\fr{\partial^2 y_+}{\partial \eta^2}(\tau;\xi,\eta)=&\exp\bigg(\dps\int_{\xi}^\tau \fr{\partial \Lambda_+(u_2)}{\partial y}(s,y_+(s;\xi,\eta))\ {\rm d}s\bigg) \nonumber \\
&\times\dps\int_{\xi}^\tau \fr{\partial^2 \Lambda_+(u_2)}{\partial y^2}\cdot\fr{\partial y_+}{\partial \eta}(s,y_+(s;\xi,\eta))\ {\rm d}s.
\end{align}
By performing a direct calculation, one can arrive at
\begin{align}\label{3.24}
&\bigg|\fr{\pa^2}{\pa y^2}\bigg(A_1\fr{u_3-u_4}{\tau}\bigg)\bigg|=\bigg|\fr{\pa^2}{\pa y^2}\bigg(\fr{a'a}{c'u_2+c'g}\cdot\fr{u_3-u_4}{\tau}\bigg)\bigg| \nonumber \\[3pt]
\leq &\bigg|\fr{a'a}{c'u_2+c'g}\fr{u_{3yy}-u_{4yy}}{\tau}\bigg| +2\bigg|\fr{(a'a)_y(c'u_2+c'g)-a'a(c'u_2+c'g)_y}{(c'u_2+c'g)^2}\bigg|\cdot\bigg|\fr{u_{3y}-u_{4y}}{\tau}\bigg| \nonumber \\[3pt]
&+\bigg|\partial_y(\fr{(a'a)_y(c'u_2+c'g)-a'a(c'u_2+c'g)_y}{(c'u_2+c'g)^2})\bigg|\cdot \bigg|\fr{u_3-u_4}{\tau}\bigg| \nonumber \\
\leq &KM\tau+K(1+M\delta)M\tau+K(1+M\delta)^2M\tau\leq K\tau(1+M\delta)^2M.
\end{align}
Similarly, we have
\begin{align}\label{3.25}
&\qquad \bigg|\fr{\pa^2}{\pa y^2}\bigg(A_3\fr{u_1-u_2}{\tau}\bigg)\bigg|\leq K\tau(1+M\delta)^2M, \nonumber \\
&|T_{1yy}|\leq K\tau(1+M\delta)^4,\quad|T_{3yy}|\leq K\tau(1+M\delta)^4.
\end{align}
Furthermore, differentiating \eqref{3.19} with respect to $\eta$ leads to
\begin{align}\label{3.26}
\bigg|\fr{\pa^2}{\pa y^2}\Lambda_+(u_2)\bigg|\leq \bigg|\fr{\partial_{yy}(c'u_2+c'g)}{(c'u_2+c'g)^2}\tau\bigg|+\bigg|\fr{2[(c'u_2+c'g)_y]^2}{(c'u_2+c'g)^2}\tau\bigg|
\leq K\tau(1+M\delta)^2,
\end{align}
which combined with \eqref{3.20} and \eqref{3.23a} acquires
\begin{align}\label{3.27}
\bigg|\fr{\partial^2 y_+}{\partial \eta^2}\bigg|\leq & e^{K\delta^2(1+M\delta)}\dps\int_{0}^\xi\bigg|\fr{\partial^2}{\partial y^2}\Lambda_+(u_2)\bigg|\cdot e^{K\delta^2(1+M\delta)}\ {\rm d}s \nonumber \\
\leq &e^{K\delta^2(1+M\delta)}\dps\int_{0}^\xi Ks(1+M\delta)^2\ {\rm d}s\leq \xi^2K(1+M\delta)^2e^{K\delta^2(1+M\delta)}.
\end{align}
Inserting \eqref{3.24}-\eqref{3.27} into \eqref{3.23}, one obtains
\begin{align}\label{3.28}
&\bigg|\fr{\partial^2 U_1}{\partial \eta^2}\bigg| +\bigg|\fr{\partial^2 U_3}{\partial \eta^2}\bigg|\leq \dps\int_{0}^\xi\bigg\{|I_1|\cdot\bigg|\fr{\partial y_+}{\partial \eta}\bigg|^2+|I_2|\bigg|\fr{\partial^2 y_+}{\partial \eta^2}\bigg| \bigg\}\bigg|\ {\rm d}\tau \nonumber \\[3pt]
\leq &\dps\int_{0}^\xi\bigg\{\bigg[\fr{1}{2}M\tau+|\lambda| KM\tau(1+M\delta)^2+|\lambda| KM\tau(1+M\delta)^2+K\tau(1+M\delta)^4 \nonumber \\[3pt]
&\qquad +K\tau(1+M\delta)^4\bigg]e^{K\delta^2(1+M\delta)}
+\bigg[\fr{1}{2}M\tau+|\lambda| KM\tau(1+M\delta)+|\lambda| KM\tau(1+M\delta) \nonumber \\[3pt]
&\qquad +K\tau(1+M\delta)^3+K\tau(1+M\delta)^3\bigg]K\delta^2(1+M\delta)^2e^{K\delta^2(1+M\delta)}\bigg\}\ {\rm d}\tau \nonumber \\[3pt]
\leq &\xi^2 e^{K\delta^2(1+M\delta)}\bigg(1+K\delta^2(1+M\delta)^2\bigg)\bigg(\fr{M}{4}+|\lambda| KM(1+M\delta)^2+K(1+M\delta)^4\bigg).
\end{align}
Doing the same procedure for $\pa_{\eta\eta}U_{2}/\xi^2$ and $\pa_{\eta\eta}U_{4}/\xi^2$ yields
\begin{align*}
&\bigg|\fr{\partial^2 U_2}{\partial \eta^2}\bigg| +\bigg|\fr{\partial^2 U_4}{\partial \eta^2}\bigg| \\
\leq &\xi^2 e^{K\delta^2(1+M\delta)}\bigg(1+K\delta^3(1+M\delta)^2\bigg)\bigg(\fr{M}{4}+|\lambda| KM(1+M\delta)^2+K(1+M\delta)^4\bigg).
\end{align*}
Thus we have
\begin{align}\label{3.29}
&\dps\sum^{4}_{j=1}\bigg |\fr{U_{j\eta\eta}(\xi,\eta)}{\xi^2}\bigg | \nonumber \\
\leq &Me^{K\delta^2(1+M\delta)}\bigg(1+K\delta^3(1+M\delta)^2\bigg)\bigg(\fr{1}{2}+|\lambda| K(1+M\delta)^2+\fr{K}{M}(1+M\delta)^4\bigg).
\end{align}

We now choose $M\geq64K\geq64$ and $\lambda_0\leq1/(32K)$ and then let $|\lambda|\leq\lambda_0$ and $\delta\leq\min\{1/M,\delta_0\}$ to get
\begin{align*}
&e^{K\delta^2(1+M\delta)}\bigg(1+K\delta^3(1+M\delta)^2\bigg)\bigg(\fr{1}{2}+|\lambda| K(1+M\delta)^2+\fr{K}{M}(1+M\delta)^4\bigg) \\
\leq & e^{\fr{\delta}{32}}\bigg(1+\fr{\delta^2}{16}\bigg)\bigg(\fr{1}{2}+\fr{1}{8}+\fr{1}{4}\bigg)\leq1.
\end{align*}
Then it follows by \eqref{3.12}, \eqref{3.22} and \eqref{3.29} that $(P_2)$-$(P_4)$ are preserved by the mapping $\mathcal{T}$. To determine $\mathcal{T}(F)\in\mathcal{F}$, it suffices to show by \eqref{3.11a} that $U_{i\xi}(0,\eta)=0, \ (i=1,2,3,4)$. To this end, we differentiate \eqref{3.4} with respect to $\xi$ to achieve
\begin{align}\label{3.30}
\fr{\partial U_1}{\partial \xi}=&\fr{u_1-u_2}{2\xi}+\lambda A_1\fr{u_3-u_4}{\xi}+T_1 \nonumber \\
&+\dps\int_{0}^\xi\bigg\{\fr{u_{1y}-u_{2y}}{2\tau}+\lambda\fr{\pa }{\pa y} \bigg(A_1\fr{u_3-u_4}{\tau}\bigg)+T_{1y}\bigg\}\fr{\partial y_+}{\partial \xi} {\rm d}\tau,
\end{align}
where
\begin{align}\label{3.31}
\fr{\partial y_+}{\partial \xi}(\tau;\xi,\eta)=-\Lambda_+(u_2)\fr{\partial y_+}{\partial \eta}(\tau;\xi,\eta).
\end{align}
From \eqref{3.30}-\eqref{3.31} and the properties $(P_2)$-$(P_4)$, it is easy to know that $U_{1\xi}(0,\eta)=0$.
Similarly, we also have $U_{i\xi}(0,\eta)=0, (i=2,3,4)$, which mean that the map $\mathcal{T}$ does map $\mathcal{F}$ into itself.

Next we check that \eqref{3.6} holds for some positive constant $\kappa<1$. By \eqref{3.3} we have
\begin{align*}
\fr{{\rm d}}{{\rm d}_+(u_2)}U_1+\fr{{\rm d}}{{\rm d}_+(u_2)}U_3=&\fr{u_1-u_2}{2\tau}+\fr{u_3-u_4}{2\tau}+\lambda A_1(\tau,y,u_2)\fr{u_3-u_4}{\tau} \\
&+\lambda A_3(\tau,y,u_2)\fr{u_1-u_2}{\tau}+T_1(\tau,y,\textbf{u})+T_3(\tau,y,\textbf{u}),  \\
\fr{{\rm d}}{{\rm d}_+(\hat{u}_2)}\hat{U}_1+\fr{{\rm d}}{{\rm d}_+(\hat{u}_2)}\hat{U}_3 =&\fr{\hat{u}_1-\hat{u}_2}{2\tau}+\fr{\hat{u}_3-\hat{u}_4}{2\tau}+\lambda A_1(\tau,y,\hat{u}_2)\fr{\hat{u}_3-\hat{u}_4}{\tau}\\
&+\lambda A_3(\tau,y,\hat{u}_2)\fr{\hat{u}_1-\hat{u}_2}{\tau}+T_1(\tau,y,\hat{\textbf{u}})+T_3(\tau,y,\hat{\textbf{u}}),
\end{align*}
from which and \eqref{3.2} we find that
\begin{align}\label{3.32}
\fr{d}{d_+(u_2)}(U_1-\hat{U}_1)+\fr{d}{d_+(u_2)}(U_3-\hat{U}_3)
=I_3+I_4+I_5+I_6+I_7+I_8,
\end{align}
where
\begin{align*}
I_3&=\fr{(u_1-\hat{u}_1)-(u_2-\hat{u}_2)}{2\tau}+\fr{(u_3-\hat{u}_3)-(u_4-\hat{u}_4)}{2\tau}, \\
I_4&=\lambda A_1(\tau,y,u_2)\fr{u_3-u_4}{\tau}-\lambda A_1(\tau,y,\hat{u}_2)\fr{\hat{u}_3-\hat{u}_4}{\tau}, \\
I_5&=\lambda A_3(\tau,y,u_2)\fr{u_1-u_2}{\tau}-\lambda A_3(\tau,y,\hat{u}_2)\fr{\hat{u}_1-\hat{u}_2}{\tau}, \\
I_6&=T_1(\tau,y,\textbf{u})-T_1(\tau,y,\hat{\textbf{u}}), \\
I_7&=T_3(\tau,y,\textbf{u})-T_3(\tau,y,\hat{\textbf{u}}), \\
I_8&=[\Lambda_+(\hat{u}_2)-\Lambda_+(u_2)](\hat{U}_{1y}+\hat{U}_{3y})
\end{align*}
It is clear that
\begin{align}\label{3.33}
|I_3|\leq \fr{|u_1-\hat{u}_1|+|u_2-\hat{u}_2|+|u_3-\hat{u}_3|+|u_4-\hat{u}_4|}{2\tau}\leq \fr{\tau}{2}d(\textbf{u},\hat{\textbf{u}}),
\end{align}
and
\begin{align}\label{3.34}
|I_8|&\leq|\Lambda_+(\hat{u}_2)-\Lambda_+(u_2)|\cdot(|\hat{U}_{1y}|+|\hat{U}_{3y}|) \nonumber \\
&\leq|\Lambda_{+u_2}|\cdot|\hat{u}_2-u_2|\cdot(|\hat{U}_{1y}|+|\hat{U}_{3y}|) \leq KM\tau^5d(\textbf{u},\hat{\textbf{u}}).
\end{align}
For the term $I_4$, we have
\begin{align}\label{3.35}
|I_2|=&|\lambda|\cdot\bigg| A_1(\tau,y,u_2)\fr{u_3-u_4}{\tau}- A_1(\tau,y,\hat{u}_2)\fr{\hat{u}_3-\hat{u}_4}{\tau}\bigg| \nonumber \\
\leq &|\lambda|\cdot |A_{1u_2}(u_2-\hat{u}_2)|\cdot\bigg|\fr{u_3-u_4}{\tau}\bigg|
+|\lambda|\cdot|A_1(\tau,y,\hat{u}_2)|\cdot\bigg|\fr{u_3-\hat{u}_3+\hat{u}_4-u_4}{\tau}\bigg| \nonumber \\
\leq &|\lambda| KM\tau^3d(\textbf{u},\hat{\textbf{u}})+|\lambda| K\tau d(\textbf{u},\hat{\textbf{u}})\leq |\lambda| K\tau(1+M\delta)d(\textbf{u},\hat{\textbf{u}}).
\end{align}
Similarly, one has
\begin{align}\label{3.36}
|I_5|\leq |\lambda| K\tau(1+M\delta)d(\textbf{u},\hat{\textbf{u}}).
\end{align}
We next estimate the terms $I_6$ and $I_7$. It is obvious that
\begin{align*}
|T_{iu_j}|\leq K(1+M\delta)^2,\ (i=1,2;\ j=1,2,3,4),
\end{align*}
from which we obtain
\begin{align}\label{3.38}
|I_4|
\leq &|T_{1u_1}(u_1-\hat{u}_1)|+|T_{1u_2}(u_2-\hat{u}_2)|+|T_{1u_3}(u_3-\hat{u}_3)| +|T_{1u_4}(u_4-\hat{u}_4)|\nonumber\\
\leq &K(1+M\delta)^2\tau^2d(\textbf{u},\hat{\textbf{u}}),
\end{align}
and
\begin{align}\label{3.39}
|I_5|
\leq &|T_{2u_1}(u_1-\hat{u}_1)|+|T_{2u_2}(u_2-\hat{u}_2)|+|T_{2u_3}(u_3-\hat{u}_3)| +|T_{2u_4}(u_4-\hat{u}_4)|\nonumber\\
\leq &K(1+M\delta)^2\tau^2d(\textbf{u},\hat{\textbf{u}}).
\end{align}
Combining \eqref{3.32}-\eqref{3.39} gets
\begin{align}\label{3.40}
|U_1-\hat{U}_1|+|U_3-\hat{U}_3|\leq &\dps \int_{0}^\tau\sum_{i=3}^8|I_i|{\rm d}\tau   \nonumber \\
\leq &\dps\int_{0}^\tau\tau\bigg\{\fr{1}{2}+|\lambda| K(1+M\delta)+|\lambda| K(1+M\delta)+K(1+M\delta)^2\delta \nonumber \\
&\qquad \quad +K(1+M\delta)^2\delta+KM\delta^4\bigg\}d(\textbf{u},\hat{\textbf{u}}){\rm d}\tau \nonumber \\
\leq &\tau^2\bigg\{\fr{1}{2}+|\lambda| K(1+M\delta)+K(1+M\delta)^2\delta+KM\delta^4\bigg\}d(\textbf{u},\hat{\textbf{u}}).
\end{align}
In view of the same argument as above, one acquires
\begin{align}\label{3.41}
|U_2-\hat{U}_2|+|U_4-\hat{U}_4|\leq\fr{1}{2}\tau^2\big\{\fr{1}{2}+|\lambda| K(1+M\delta)+K(1+M\delta)^2\delta+KM\delta^4\big\}d(\textbf{u},\hat{\textbf{u}}).
\end{align}
Adding \eqref{3.40} and \eqref{3.41} leads to
\begin{align*}
\dps\sum_{j=1}^4\fr{|U_i-\hat{U}_i|}{\tau^2}\leq\bigg\{\fr{1}{2}+|\lambda| K(1+M\delta)+K(1+M\delta)^2\delta+KM\delta^4\bigg\}d(\textbf{u},\hat{\textbf{u}}) :=\kappa d(\textbf{u},\hat{\textbf{u}}).
\end{align*}
for $\kappa<1$ if $|\lambda|\leq\lambda_0$ and $\delta$ is chosen as before, which finishes the proof of \eqref{3.6}. Hence $\mathcal{T}$ is a contraction under the metric $d$, and the proof of the lemma is complete.
\end{proof}

\textbf{Step 3: Properties of the limit function.} We assert that the limit of the iteration sequence $\{\textbf{u}^{(n)}\}$, defined by $\textbf{u}^{(n)}=\mathcal{T}\textbf{u}^{(n-1)}$, is also in $\mathcal{F}$.
This assertion is delivered directly by the following lemma.
\begin{lem}\label{lem2}
Let the assumptions in Theorem \ref{thm3} hold. Then, for $|\lambda|\leq\lambda_0$, the iteration sequence $\{\textbf{u}^{(n)}\}$ satisfies that $\{\partial_\tau \textbf{u}^{(n)}(\tau,y)\}$ and $\{\partial_y \textbf{u}^{(n)}(\tau,y)\}$
are uniformly Lipschitz continuous on $[0,\delta]\times R$.
\end{lem}
\begin{proof}
Assume that $\textbf{u}=(u_1, u_2, u_3, u_4)^T\in\mathcal{F}$. Thanks to Lemma \ref{lem1} we know that $\textbf{U}=(U_1, U_2, U_3, U_4)^T=\mathcal{T}(\textbf{u})$ also in $\mathcal{F}$. We next derive, in turn, the estimates of the terms
$$
\dps\sum_{i=1}^4\bigg|\fr{\pa U_i}{\pa\xi}(\xi,\eta)\bigg|,\quad \sum_{i=1}^4\bigg|\fr{\pa^2 U_i}{\pa\xi\pa\eta}(\xi,\eta)\bigg|, \quad {\rm and}\quad  \sum_{i=1}^4\bigg|\fr{\pa^2 U_i}{\pa\xi^2}(\xi,\eta)\bigg|.
$$

Firstly, we use the estimates in Step 2 to obtain
\begin{align}\label{3.42}
\bigg|\fr{\partial U_1}{\partial \xi}\bigg|+\bigg|\fr{\partial U_3}{\partial \xi}\bigg| \leq &\bigg|\fr{u_{1}-u_2}{2\xi}\bigg|+\bigg|\fr{u_3-u_4}{2\xi}\bigg|
+\bigg|\lambda A_1\fr{u_3-u_4}{\xi}\bigg|+\bigg|\lambda A_3\fr{u_1-u_2}{\xi}\bigg|+|T_1|+|T_3| \nonumber \\[3pt]
&+\dps\int_{0}^\xi\bigg\{\bigg|\fr{u_{1y}-u_{2y}}{2\tau}\bigg| +\bigg|\fr{u_{3y}-u_{4y}}{2\tau}\bigg|+\bigg|\lambda\fr{\pa}{\pa y}\bigg(A_1\fr{u_3-u_4}{\tau}\bigg)\bigg| \nonumber \\[3pt]
&\qquad \qquad +\bigg|\lambda\fr{\pa}{\pa y}\bigg(A_3\fr{u_1-u_2}{\tau}\bigg)\bigg|+|T_{1y}|+|T_{3y}|\bigg\}\bigg|\fr{\partial y_+}{\partial \xi}\bigg| {\rm d}\tau  \nonumber \\[3pt]
\leq &\fr{1}{2}M\xi+\lambda_0 KM\xi+K(1+M\delta)^2\xi \nonumber \\[3pt]
&\  +\dps\int_{0}^\xi\bigg\{\fr{1}{2}M\tau+\lambda_0 KM(1+M\delta)\tau+K(1+M\delta)^3\tau\bigg\}K\tau e^{K\delta^2(1+M\delta)}{\rm d}\tau\nonumber \\[3pt]
\leq &\bigg\{\fr{M}{2}+\lambda_0 KM+K(1+M\delta)^2  \nonumber \\[3pt]
&+K \delta^2\bigg(M+\lambda_0 KM(1+M\delta)+K(1+M\delta)^3\bigg)e^{K\delta^2(1+M\delta)}\bigg\}\xi \nonumber \\[3pt]
\leq & M\xi,
\end{align}
if $M$, $\lambda_0$ and $\delta$ are chosen as in Lemma \ref{lem1}. The above is also valid for $|U_{2\xi}|+|U_{4\xi}|$. Hence we get the estimate of the term $\sum_{i=1}^4|U_{i\xi}|$
\begin{align}\label{3.43}
\dps\sum_{i=1}^4\bigg|\fr{\partial U_i}{\partial\xi}(\xi,\eta)\bigg|\leq 2M\xi.
\end{align}

Secondly, we calculate by differentiating \eqref{3.13} with respect to $\xi$
\begin{align}\label{3.44}
\fr{\partial}{\partial\xi}\bigg(\fr{\partial U_1}{\partial \eta} (\xi,\eta)\bigg)+\fr{\partial}{\partial\xi}\bigg(\fr{\partial U_3}{\partial \eta}(\xi,\eta)\bigg)=I_{9} +\dps\int_{0}^\xi\bigg\{I_{10}\fr{\partial y_+}{\partial\eta}\cdot\fr{\partial y_+}{\partial\xi} +I_{11}\fr{\partial^2 y_+}{\partial \eta\partial\xi}\bigg\}\ {\rm d}\tau,
\end{align}
where
\begin{align*}
I_9=&\fr{u_{1\eta}-u_{2\eta}}{2\xi}+\fr{u_{3\eta}-u_{4\eta}}{2\xi} +\lambda\fr{\pa}{\pa \eta}\bigg(A_1\fr{u_3-u_4}{\xi} +A_3\fr{u_1-u_2}{\xi}\bigg)+T_{1\eta}+T_{3\eta}, \\[3pt]
I_{10}=&\fr{u_{1yy}-u_{2yy}}{2\tau}+\fr{u_{3yy}-u_{4yy}}{2\tau}
+\lambda\fr{\pa^2}{\pa y^2}\bigg(A_1\fr{u_3-u_4}{\tau}+A_3\fr{u_1-u_2}{\tau}\bigg)
+T_{1yy}+T_{3yy}, \\[3pt]
I_{11}=&\fr{u_{1y}-u_{2y}}{2\tau}
+\fr{u_{3y}-u_{4y}}{2\tau}+\lambda \fr{\pa}{\pa y}\bigg(A_1\fr{u_3-u_4}{\tau}+A_3\fr{u_1-u_2}{\tau}\bigg)
+T_{1y}+T_{3y}
\end{align*}
and
\begin{align*}
\fr{\partial^2y_+}{\partial\xi\partial\eta}(\tau;\xi,\eta) =\exp\bigg(\dps\int_{\xi}^\tau\fr{\partial\Lambda_+(u_2)}{\partial y}\bigg){\rm d}s\cdot
\bigg\{\dps\int_{\xi}^\tau\fr{\partial^2\Lambda_+(u_2)}{\partial y^2}\fr{\partial y_+}{\partial\xi}{\rm d}s-\fr{\partial\Lambda_+(u_2)}{\partial y}\bigg\}.
\end{align*}
It follows by combining \eqref{3.19}, \eqref{3.26} and using the relation \eqref{3.31} that
\begin{align}\label{3.45}
\big|\fr{\partial^2y_+}{\partial\xi\partial\eta}\big|\leq &e^{K\delta^2(1+M\delta)}\bigg\{\dps\int_{0}^\delta Ks(1+M\delta)^2se^{K\delta^2(1+M\delta)}{\rm d}s+K\delta(1+M\delta)\bigg\} \nonumber \\
 \leq &e^{K\delta^2(1+M\delta)}K\delta(1+M\delta)^2.
\end{align}
One employs the estimates in Step 2 again to achieve
\begin{align*}
|I_9|\leq &\fr{1}{2}M\xi+\lambda_0KM\xi(1+M\delta)+K\xi(1+M\delta)^3, \\
|I_{10}|\leq &\fr{1}{2}M\tau+\lambda_0 KM\tau(1+M\delta)^2+K\tau(1+M\delta)^4, \\
|I_{11}|\leq & \fr{1}{2}M\tau+\lambda_0 KM\tau(1+M\delta) +K\tau(1+M\delta)^3.
\end{align*}
Inserting the above and \eqref{3.45} into \eqref{3.44} yields
\begin{align}\label{3.46}
&\bigg|\fr{\partial^2 U_1}{\partial\eta\partial\xi}\bigg|+\bigg|\fr{\partial^2 U_3}{\partial \eta\partial\xi}\bigg|\leq\fr{1}{2}M\xi+\lambda_0 KM\xi(1+M\delta)+K\xi(1+M\delta)^3 \nonumber \\[3pt]
&\quad +\dps\int_{0}^\xi\bigg\{\bigg(\fr{1}{2}M\tau+\lambda_0 KM\tau(1+M\delta)^2+K\tau(1+M\delta)^4\bigg)e^{K\delta(1+M\delta)}K\tau \nonumber \\[3pt]
&\quad +\bigg(\fr{1}{2}M\tau+\lambda_0 KM\tau(1+M\delta)+K\tau(1+M\delta)^3\bigg)K\delta(1+M\delta)^2e^{K\delta(1+M\delta)}\bigg\}{\rm d}\tau
\nonumber \\[3pt]
\leq&\bigg\{\fr{1}{2}M+\lambda_0 KM(1+M\delta)+K(1+M\delta)^3  \nonumber \\[3pt]
&\ \ +K\delta^2\bigg(M+\lambda KM(1+M\delta)^2+K(1+M\delta)^4\bigg) e^{K\delta(1+M\delta)}\bigg\}\xi\leq M\xi,
\end{align}
by the chosen of $M$, $\lambda_0$ and $\delta$ as in Step 2. The same estimate of \eqref{3.46} is also true for $|U_{2\xi\eta}+U_{4\xi\eta}|$. Therefore we have
\begin{align}\label{3.47}
\dps\sum_{i=1}^4\bigg|\fr{\partial^2 U_i}{\partial\xi\partial\eta}(\xi,\eta)\bigg|\leq 2M\xi.
\end{align}

Finally, we derive the bound of $\sum_{i=1}^4|U_{i\xi\xi}|$. By differentiating \eqref{3.42} with respect to $\xi$, one obtains
\begin{align}\label{3.48}
\fr{\partial}{\partial\xi}\bigg(\fr{\partial U_1}{\partial \xi}(\xi,\eta)\bigg)+\fr{\partial}{\partial\xi}\bigg(\fr{\partial U_3}{\partial \xi}(\xi,\eta)\bigg)= I_{12} +\dps\int_{0}^\xi\bigg\{ I_{13}\bigg(\fr{\partial y_+}{\partial\xi}\bigg)^2 +I_{14}\fr{\partial^2 y_+}{\partial\xi^2} \bigg\}\ {\rm d}\tau,
\end{align}
where
\begin{align*}
I_{12}=&\fr{u_{1\xi}-u_{2\xi}}{\xi}+\fr{u_{3\xi}-u_{4\xi}}{\xi}-\fr{u_1-u_2}{2\xi^2}-\fr{u_3-u_4}{2\xi^2} \\[3pt] &+2\lambda\fr{\pa}{\pa \xi}\bigg(A_1\fr{u_3-u_4}{\xi} +A_3\fr{u_1-u_2}{\xi}\bigg)
+2T_{1\xi}+2T_{3\xi}, \\[3pt]
I_{13}=& \fr{u_{1yy}-u_{2yy}}{2\tau}+\fr{u_{3yy}-u_{4yy}}{2\tau} +\lambda\fr{\pa^2}{\pa y^2}\bigg(A_1\fr{u_3-u_4}{\tau} +A_3\fr{u_1-u_2}{\tau}\bigg)+T_{1yy}+T_{3yy}, \\[3pt]
I_{14}=& \fr{u_{1y}-u_{2y}}{2\tau}
+\fr{u_{3y}-u_{4y}}{2\tau}+\lambda\fr{\pa}{\pa y}\bigg(A_1\fr{u_3-u_4}{\tau} +A_3\fr{u_1-u_2}{\tau}\bigg) +T_{1y}+T_{3y}.
\end{align*}
By a direct calculation, we acquire
\begin{align}\label{3.50}
&\bigg|\partial_\xi\bigg(A_1(\xi,\eta,u_2,u)\fr{u_3-u_4}{\xi}\bigg)\bigg| \leq |\pa_\xi A_{1}|\fr{|u_3-u_4|}{\xi}+|A_1|\bigg(\fr{|u_{3\xi}-u_{4\xi}|}{\xi} +\fr{|u_3-u_4|}{\xi^2}\bigg)  \nonumber   \\[3pt]
\leq &\bigg(|A_{1\xi}|+|A_{1u_2}||u_{2\xi}|+|A_{1u}||u_{\xi}|\bigg)\fr{|u_3-u_4|}{\xi}  +|A_1|\bigg(\fr{|u_{3\xi}-u_{4\xi}|}{\xi} + \fr{|u_3-u_4|}{\xi^2}\bigg) \nonumber
\\[3pt]
\leq&[K+KM\xi+K(1+M\xi)]M\xi+KM+KM\leq K(1+M\delta)M\delta+KM.
\end{align}
Similarly, one arrives at
\begin{align}\label{3.51}
\bigg|\partial_\xi\bigg(A_3(\xi,\eta,u_2,u)\fr{u_1-u_2}{\xi}\bigg)\bigg|\leq K(1+M\delta)M\delta+KM.
\end{align}
Moreover, we can derive the following estimates
\begin{align*}
\bigg|\partial_\xi\bigg(\fr{1}{4(u_2+g)}\fr{(u_1-u_2)^2}{\xi}\bigg)\bigg|\leq &K(1+M\delta)M^2\xi^3+KM^2\xi^2,  \\
\dps\sum^{4}_{j=1}|(T_{1j}u_{j})_\xi|\leq&\dps\sum^{4}_{j=1}|T_{1j\xi}||u_j|+\dps\sum^{4}_{j=1}|T_{1j}||u_{j\xi}|  \\
\leq &K(1+M\delta)M\xi^2+K(1+M\delta)M\xi^2+K(1+M\delta)M\xi^2  \\
&\ \ +K(1+M\delta)^2M\xi^2+K(1+M\delta)M\xi \\
\leq &KM\delta(1+M\delta)^3, \\
|(F_{1\xi})_\xi|\leq& K(1+M\delta)^2\delta+K(1+M\delta),
\end{align*}
from which one has
\begin{align}\label{3.53}
|T_{1\xi}|\leq&\bigg|\partial_\xi\bigg(\fr{1}{4(u_2+g)}\fr{(u_1-u_2)^2}{\xi}\bigg)\bigg| +\dps\sum^{4}_{j=1}|T_{1j\xi}||u_j|+\dps\sum^{4}_{j=1}|T_{1j}||u_{j\xi}|+|(F_{1\xi})_{\xi}|\nonumber \\
\leq &K\delta(1+M\delta)^4+K(1+M\delta).
\end{align}
By the same argument for $T_3$ as above we obtain
\begin{align}\label{3.54}
|T_{3\xi}|\leq K\delta(1+M\delta)^4+K(1+M\delta).
\end{align}
Combining \eqref{3.50}-\eqref{3.54} and using the estimates in Step 2 lead to
\begin{align*}
|I_{12}|\leq &3M+\lambda_0 [K(1+M\delta)M\delta+KM]+KM\delta(1+M\delta)^4+K(1+M\delta),\\
|I_{13}|\leq &\fr{1}{2}M\tau+\lambda_0 K\tau(1+M\delta)^2M+K\tau(1+M\delta)^4,\\
|I_{14}|\leq &\fr{1}{2}M\tau+\lambda_0 K\tau(1+M\delta)M+K(1+M\delta)^3.
\end{align*}
Putting the above into \eqref{3.48} and applying the estimate
\begin{align}
\bigg|\fr{\partial^2y_+}{\partial\xi^2}\bigg|=&\bigg|\fr{\partial\Lambda_+(u_2)}{\partial \xi}\cdot\fr{\partial y_+}{\partial \eta}+\Lambda_+(u_2)\fr{\partial^2y_+}{\partial\xi\partial\eta}\bigg| \nonumber \\[3pt]
\leq &K[1+M\delta^2+\delta(1+M\delta)]e^{K\delta^2(1+M\delta)}+K\delta\cdot K\delta(1+M\delta)^2e^{K\delta^2(1+M\delta)} \nonumber \\
\leq &K(1+M\delta)^3e^{K\delta^2(1+M\delta)}, \nonumber
\end{align}
one gets
\begin{align}\label{3.56}
&\bigg|\fr{\partial^2U_1}{\partial\xi^2}\bigg|+\bigg|\fr{\partial^2U_3}{\partial\xi^2}\bigg|\leq3M+\lambda_0 [K(1+M\delta)M\delta+KM]+KM\delta(1+M\delta)^4+K(1+M\delta) \nonumber \\
&\ \ +e^{K\delta^2(1+M\delta)}K(1+M\delta)^3\delta^2\bigg(M+\lambda KM(1+M\delta)^2+K(1+M\delta)^4\bigg) \nonumber \\
\leq &8M,
\end{align}
by the chosen of $M$, $\lambda_0$ and $\delta$ as in Lemma \ref{lem1}. The estimate in \eqref{3.56} also holds for
$|U_{2\xi\xi}|+|U_{4\xi\xi}|$. Thus we have
\begin{align}\label{3.57}
\dps\sum_{i=1}^4\bigg|\fr{\partial^2U_i}{\partial\xi^2}\bigg|\leq 16M.
\end{align}
Summing up \eqref{3.43}, \eqref{3.47} and \eqref{3.57}, we finish the proof of Lemma \ref{lem2}.
\end{proof}

Based on Lemmas \ref{lem1} and \ref{lem2}, the proof of Theorem \ref{thm3} is now completed.

\textbf{Proof of Theorems 1.1 and 1.2.} According to \eqref{2.8} and Theorem \ref{thm3}, we first obtain the functions $(R_i,S_i)(y,\tau)\ (i=1,2)$. Moreover, recalling the Jacobian
\begin{align*}
J:=\fr{\partial(\tau,y)}{\partial(t,x)}=-c'\fr{R_1+S_1}{2}=-c'\bigg(\fr{U_1+U_2+2g_{11}\tau}{2}+\psi_1\bigg)\geq \fr{m_0\psi_0}{2}>0,
\end{align*}
we see that the transformation $(x, t)\rightarrow(y, \tau)$ is a one-to-one mapping. Therefore one can acquire $(R_i,S_i)\ (i=1,2)$ as smooth functions of $t$ and $x$. We integrate the equations for $u$ and $v$ in \eqref{1.12} to get the functions $u(t,x)$ and $v(t,x)$, respectively. It is not difficult to check that the functions $(R_1, S_1, R_2, S_2, u, v)$ defined above satisfy system \eqref{1.12} and the initial conditions \eqref{1.13}, which
ends the proof of Theorem \ref{thm1}. To show Theorem \ref{thm2}, is suffices to verify that $u_x=(R_1-S_1)/2c$ and $v_x=(R_2-S_2)/2c$ hold in $[0,\delta]\times R$. Denote
$$
H_1=R_1-S_1-2cu_x,\quad H_2=R_2-S_2-2cv_x.
$$
It is easily seen that $H_1(x,0)=H_2(x,0)=0$ by \eqref{1.13}. Furthermore, we directly calculate by \eqref{1.12} to find that
\begin{align}\label{3.58}
\partial_tH_i=\fr{c'(R_1+S_1)}{2}\cdot\fr{H_i}{c},\ (i=1,2).
\end{align}
Making use of \eqref{2.3}, the equations \eqref{3.58} can be rewritten as
\begin{align*}
\partial_\tau H_i=\fr{H_i}{\tau},\ (i=1,2).
\end{align*}
We note that $H_i/\tau=0,\ (i=1,2)$ on the line $\tau=0$, which implies $H_1(y,\tau)=H_2(y,\tau)\equiv0$ in $(0,\delta]\times R$. Thus we have $H_1(x,t)=H_2(x,t)\equiv0$ and the proof of Theorem \ref{thm2} is complete.

\section*{Acknowledgements}

This work was supported by the Zhejiang Provincial Natural Science Foundation (No. LY17A010019).

\section*{Appendix}
\appendix

\section{The expressions of $T_{ij}$ and $F_i$ in \eqref{2.11}}\label{app}

We here list the detailed expressions of $T_{ij}$ and $F_i$ in \eqref{2.11}.
\begin{align*}
T_{11}&=0,\quad \qquad  \qquad  \qquad  \qquad \
T_{12}=\fr{aa'\psi_2^2+2\lambda aa'\varphi'_2}{c'(\varphi_1)\psi_1(y)(U_2+g)}, \\
T_{13}&=\fr{c'a^2\varphi'_2-aa'\psi_2}{c'(U_2+g)},\quad \qquad \
T_{14}=-\fr{c'a^2\varphi'_2+aa'\psi_2+ aa'U_3}{c'(U_2+g)},  \\
T_{21}&=\fr{aa'\psi^2_2+2\lambda aa'\varphi'_2}{c'(\varphi_1)\psi_1(y)(U_1+g)},\quad
T_{22}=0,\\
T_{23}&=\fr{a^2c'\varphi'_2-aa'\psi_2}{c'(U_1+g)},\quad \qquad \
T_{24}=-\fr{c'a^2\varphi'_2+aa'\psi_2+ aa'U_3}{c'(U_1+g)},
\end{align*}
\begin{align*}
T_{31}&=\fr{a'\psi_2-ac'\varphi'_2}{ac'(U_2+g)},\quad
T_{32}=\fr{a'\psi_2+ac'\varphi'_2}{ac'(U_2+g)}
-\fr{2a'\psi_1\psi_2}{a(\varphi_1)c'(\varphi_1)\psi_1(y)(U_2+g)},\\
T_{33}&=\fr{a'(U_2+\psi_1)}{ac'(U_2+g)},\quad
\ T_{34}=\fr{a'(U_1+ \psi_1)}{ac'(U_2+g)}, \\
T_{41}&=\fr{a'\psi_2-ac'\varphi'_2}{ac'(U_1+g)} -\fr{2a'\psi_1\psi_2}{a(\varphi_1)c'(\varphi_1)\psi_1(y)(U_1+g)},\quad
T_{42}=\fr{a'\psi_2+ac'\varphi'_2}{ac'(U_1+g)}, \\
T_{43}&=\fr{a'(U_2+\psi_1)}{ac'(U_1+g)},\quad
T_{44}=\fr{a'(U_1+\psi_1)}{ac'(U_1+g)},
\end{align*}

\begin{align*}
F_1=&-\fr{c'a^2(\varphi'_2)^2+aa'\psi_2(g_{21}+g_{22})}{c'(U_2+g)}
+\fr{\psi'_1+g'_{11}\tau- aa'g_{21}g_{22}\tau}{c'(U_2+g)} \\ \
&-\fr{aa'(g_{22}U_3+g_{21}U_4)}{c'(U_2+g)}
+\fr{(aa'\psi^2_2+2\lambda aa'\varphi'_2)g_{11}}{c'(\varphi_1)\psi_1(y)(U_2+g)}  \\
&\ +\fr{aa'\psi^2_2+2\lambda aa'\varphi'_2}{c'(\varphi_1)c'(U_2+g)}\fr{c'(u)-c'(\varphi_1)}{\tau}
+\fr{2\lambda\varphi'_2+\psi^2_2}{2c'(\varphi_1)\psi_1(y)}\fr{(a^2)'(\varphi_1)-(a^2)'(u)}{\tau},
\end{align*}
\begin{align*}
F_2=&-\fr{c'a^2(\varphi'_2)^2+aa'\psi_2(g_{21}+g_{22})}{c'(U_1+g)}
+\fr{\psi'_1+g'_{11}\tau- aa'g_{21}g_{22}\tau}{c'(U_1+g)} \\
&\ -\fr{aa'(g_{22}U_3+g_{21}U_4)}{c'(U_1+g)} +\fr{(aa'\psi^2_2+2\lambda aa'\varphi'_2)g_{11}}{c'(\varphi_1)\psi_1(y)(U_1+g)} \\
&\ +\fr{aa'\psi^2_2+2\lambda aa'\varphi'_2}{c'(\varphi_1)c'(U_1+g)}\fr{c'(u)-c'(\varphi_1)}{\tau}
+\fr{2\lambda\varphi'_2+\psi^2_2}{2c'(\varphi_1)\psi_1(y)}\fr{(a^2)'(\varphi_1)-(a^2)'(u)}{\tau},
\end{align*}
\begin{align*}
F_3=&\fr{a'(g_{22}U_1 +g_{21}U_2 +g_{11}U_3 +g_{11}U_4)}{ac'(U_2+g)}+\fr{a\psi'_2+ag'_{21}\tau +a'g_{11}(g_{22}+g_{21})\tau}{ac'(U_2+g)}\\
&\ +\fr{a'(\psi_1g_{22}+\psi_2g_{11}+\psi_1g_{21}+\psi_2g_{11})}{a c'(U_2+g)} +\fr{2\psi_2(y)\psi_1(y)}{a(\varphi_1)c'(\varphi_1)\psi_1(y)}\cdot\fr{a'(u)-a'(\varphi_1)}{\tau} \\
&\ -\fr{2a'\psi_1\psi_2g_{11}}{a(\varphi_1)c'(\varphi_1)\psi_1(y)(U_2+g)}
-\fr{2a'\psi_1\psi_2}{a(c'U_2+c'g)a(\varphi_1)c'(\varphi_1)}\cdot\fr{a(u)c'(u)-a(\varphi_1)c'(\varphi_1)}{\tau},
\end{align*}
\begin{align*}
F_4=&\fr{a'(g_{22}U_1+g_{21}U_2+ g_{11}U_3+ g_{11}U_4)}{a c'(U_1+g)} -\fr{a\psi'_2+ag'_{22}\tau -a'g_{11}(g_{22}+g_{21})\tau}{ac'(U_1+g)} \\
&\ +\fr{a'(\psi_1g_{22} +\psi_2g_{11} +\psi_2g_{11}+ \psi_1g_{21})}{a c'(U_1+g)} +\fr{2\psi_1(y)\psi_2(y)}{a(\varphi_1)c'(\varphi_1)\psi_1(y)}\cdot\fr{a'(u)-a'(\varphi_1)}{\tau} \\
&\  -\fr{2a'\psi_1\psi_2g_{11}}{(U_1+g)a(\varphi_1)c'(\varphi_1)\psi_1(y)}
-\fr{2a'\psi_1\psi_2}{a(c'U_1+c'g)a(\varphi_1)c'(\varphi_1)}\cdot\fr{a(u)c'(u)-a(\varphi_1)c'(\varphi_1)}{\tau}.
\end{align*}

\end{document}